\documentclass{amsart}

\usepackage[foot]{amsaddr}
\usepackage{lmodern}
\usepackage[T1]{fontenc}
\usepackage[utf8]{inputenc}

\usepackage{amsmath,amssymb,amsfonts}
\usepackage[sort&compress,square,numbers]{natbib}
\usepackage{mathtools}

\usepackage{hyperref}
\hypersetup{
  pdfauthor={Daniel F. Scharler, Johannes Siegele, Hans-Peter Schröcker},
  pdftitle={Quadratic Split Quaternion Polynomials: Factorization and Geometry},
  pdfkeywords={},
  hidelinks,
}

\newcommand{\qi}{\mathbf{i}}
\newcommand{\qj}{\mathbf{j}}
\newcommand{\qk}{\mathbf{k}}
\newcommand{\Cj}[1]{{#1}^\ast}
\newcommand{\SPQ}{\mathbb{S}}
\newcommand{\R}{\mathbb{R}}
\newcommand{\C}{\mathbb{C}}
\newcommand{\SO}[1][3]{\mathrm{SO}(#1)}
\newcommand{\SE}[1][3]{\mathrm{SE}(#1)}
\newcommand{\N}{\mathcal{N}}
\renewcommand{\P}{\mathbb{P}}
\DeclareMathOperator{\RE}{Re}
\DeclareMathOperator{\IM}{Im}
\newcommand{\Scalar}[1]{\RE(#1)}
\newcommand{\Vector}[1]{\IM(#1)}

\newtheorem{thm}{Theorem}[section]
\newtheorem{lem}[thm]{Lemma}

\newtheorem{cor}[thm]{Corollary}
\theoremstyle{definition}
\newtheorem{defn}[thm]{Definition}
\theoremstyle{remark}
\newtheorem{rmk}[thm]{Remark}
\newtheorem{example}[thm]{Example}

\title[Quadratic Split Quaternion Polynomials]{Quadratic Split Quaternion Polynomials: Factorization and Geometry}

\date{\today}

\author{Daniel F. Scharler \and Johannes Siegele \and Hans-Peter Schröcker}
\address{Department of Basic Sciences in Engineering Sciences, University of Innsbruck, Technikerstr.~13, 6020 Innsbruck, Austria}
\email{daniel.scharler@uibk.ac.at}
\email{johannes.siegele@uibk.ac.at}
\email{hans-peter.schroecker@uibk.ac.at}

\keywords{skew polynomial ring, null quadric, Clifford translation, left/right
  ruling, zero divisor, projective geometry, non-Euclidean geometry}
\subjclass[2010]{
  12D05, 16S36, 51M09, 51M10, 70B10  }

\begin{document}

\begin{abstract}
    We investigate factorizability of a quadratic split quaternion polynomial. In
  addition to inequality conditions for existence of such factorization, we
  provide lucid geometric interpretations in the projective space over the split
  quaternions.
 \end{abstract}

\maketitle

\section{Introduction}

Quaternions and dual quaternions provide compact and simple parametrizations
for the groups $\SO$, $\SE[2]$ and $\SE[3]$. This accounts for their importance
in fields such as kinematics, robotics and mechanism science. In this context,
polynomials over quaternion rings in one indeterminate can be used to
parameterize rational motions. Factorization of polynomials corresponds to the
decomposition of a rational motion into rational motions of lower degree. Since
linear factors generically describe rotational motions, factorizations with
linear factors give rise to a sequence of revolute joints from which mechanisms
can be constructed \cite{hegedus13}.

In recent years, the theory of quaternion polynomial factorization
\cite{niven41,gordon65} has been extended to the dual quaternion case and
numerous applications have been found \cite{li15b,li15c,li18b}. The main
difficulty in comparison with the purely quaternion theory is the presence of
zero divisors. As of today our general understanding of dual quaternion
factorization is quite profound but some questions still remain. Most notably, a
complete characterization of all polynomials that admit factorizations with only
linear factors and algorithms for computing them are still missing. Both exist
for ``dense'' classes of dual quaternion polynomials \cite{li15b}.

A first step in research on factorizability and factorization algorithms of
polynomials should be the investigation of quadratic polynomials. This has been
done for quaternions in \cite{huang02} and for split quaternions in
\cite{cao19}. Results for generalized quaternions, including split quaternions,
can also be found in \cite{abrate09}. The generic case is subsumed in a generic
factorization theory as in \cite{li15b,li18a} while special cases still allow a
complete discussion.

In this article, we consider quadratic left polynomials over the split
quaternions. Factorization results for these polynomials are among the topics of
\cite{abrate09} and \cite{cao19}. We present different characterizations,
tailored towards later geometric interpretation of factorizability, based on the
geometry of split quaternions. It is much clearer than the inequality criteria
with their numerous case distinctions that had been known so far (c.~f.
Theorem~\ref{th:geometric-characterization},
Corollary~\ref{cor:independent} and Corollary~\ref{cor:vanishing}). Moreover, we also
use our criteria for covering polynomials with non-invertible leading
coefficient which hitherto have not been dealt with.

Factorization of quadratic polynomials over the split quaternions is interesting
from a purely algebraic viewpoint because, like dual quaternions but unlike
ordinary quaternions, the ring of split quaternions contains zero divisors.
Their structure is more involved than in the dual quaternion case but,
nonetheless, allows a reasonably simple computational treatment and a nice
geometric interpretation. It is also relevant in hyperbolic kinematics
\cite{li18a} and isomorphic to the fundamental algebra of real $2 \times 2$
matrices.

The remainder of this article is structured as follows: Following a presentation
of basic results on split quaternions and their geometry in
Section~\ref{sec:preliminaries}, our main results are given in
Section~\ref{sec:factorization-results}. We first derive our own inequality
conditions and their geometric interpretation for the cases of dependent and
independent coefficients under the assumption that the norm polynomial does not
vanish in Sections~\ref{sec:factor-monic-dependent} and
\ref{sec:factor-monic-independent}. A geometric interpretation of these cases is
given in Section~\ref{sec:geometric-interpretation}. The remaining case of
vanishing norm polynomial is the topic of the concluding
Section~\ref{sec:factor-non-invertible}.

\section{Preliminaries}
\label{sec:preliminaries}

\subsection{Split Quaternions and Split Quaternion Polynomials}

The algebra of split quaternions, denoted by $\SPQ$, is generated by the
quaternion units $\qi$, $\qj$ and $\qk$ over the real numbers $\R$. An element
$h \in \SPQ$ is given by $h = h_0 + h_1 \qi + h_2 \qj + h_3 \qk$, where $h_0$,
$h_1$, $h_2$, $h_3 \in \R$ are real numbers. The multiplication of split
quaternions is defined by the relations
\begin{align*}
  \qi^2 = - \qj^2 = - \qk^2 = -\qi \qj \qk = -1.
\end{align*}
From this, the complete multiplication table may be inferred and one finds that
the basis elements anti-commute, e.g. $\qi\qj = -\qj \qi$. The split quaternion
\emph{conjugate} to $h = h_0 + h_1 \qi + h_2 \qj + h_3 \qk$ is defined as
$\Cj{h} \coloneqq h_0 - h_1 \qi - h_2 \qj - h_3 \qk$. Conjugation of split
quaternions $h \mapsto \Cj{h}$ is an anti-automorphism, i.e. $\Cj{(hg)} = \Cj{g}
\Cj{h}$ for $h$, $g \in \SPQ$. The split quaternion norm is defined by $h \Cj{h}
= \Cj{h} h = h_0^2 + h_1^2 - h_2^2 - h_3^2 \in \R$. A split quaternion $h$ is
invertible if and only if $h\Cj{h} \neq 0$ in which case $h^{-1} = (h
\Cj{h})^{-1} \Cj{h}$. The \emph{scalar} or \emph{real part} of $h \in \SPQ$ is
$\Scalar{h} \coloneqq \frac{1}{2}(h + \Cj{h}) = h_0$, the \emph{vector} or
\emph{imaginary part} is $\Vector{h} \coloneqq \frac{1}{2}(h - \Cj{h}) = h_1\qi
+ h_2\qj + h_3\qk$. The split quaternion $h$ is called \emph{vectorial} if
$\Scalar{h} = 0$.

By $\SPQ[t]$ we denote the ring of polynomials in one indeterminate $t$ with
split quaternion coefficients. Addition is done in the usual way; multiplication
is defined by the convention that the indeterminate $t$ commutes with all
coefficients in $\SPQ$. This is motivated by applications in hyperbolic
kinematics \cite{li18a} where $t$ serves as a real motion parameter that,
indeed, is in the center of $\SPQ$. Consider a left polynomial $P = \sum_{\ell =
  0}^n p_\ell t ^\ell \in \SPQ[t]$ (coefficients are written to the left hand
side of the indeterminate $t$). The \emph{conjugate polynomial} $\Cj{P}
\coloneqq \sum_{\ell = 0}^n \Cj{p}_{\ell} t^\ell$ is obtained by conjugation of
the coefficients. Hence, the norm polynomial $P \Cj{P} = \Cj{P} P \in \R[t]$ is
real. The \emph{evaluation} of $P$ at $h \in \SPQ$ is defined by $P(h) \coloneqq
\sum_{\ell = 0}^n p_\ell h ^\ell$. One calls it a \emph{right evaluation}
because the variable $t$ is written to the right hand side of the coefficients
and then substituted by $h$. To illustrate the substantial difference to the
left evaluation (of right polynomials) where the variable $t$ is written to the
left hand side of the coefficients, consider the polynomial $h_1 t = t h_1 \in
\SPQ[t]$ ($t$ commutes with $h_1$) and a split quaternion $h_2 \in \SPQ$. Right
evaluation of $h_1 t$ at $h_2$ yields $h_1 h_2$ whereas left evaluation of $t
h_1$ yields $h_2 h_1$. The results are different unless $h_1$ and $h_2$ commute.

Due to non-commutativity of split quaternion multiplication we have to differ
between right and left factors and zeros of a polynomial as well. Consider two
split quaternion polynomials $P$, $F \in \SPQ[t]$. We call $F$ a right factor of
$P$ if there exists a polynomial $Q \in \SPQ[t]$ such that $P = Q F$. A right
zero $h$ of a left polynomial $P$ is defined by the property that the right
evaluation of $P$ at $h$ vanishes. Left factors and left zeros are defined
analogously. In this paper we mainly deal with left polynomials, right
evaluation, right factors and right zeros but often simply speak of polynomials,
evaluation, factors and zeros, respectively. Of course, there exists a symmetric
theory on right polynomials and left evaluation, factors and zeros.

\subsection{Geometry of Split Quaternions}

In this section we take a look at the geometry of split quaternions which, as we
shall see, is closely related to factorizability of split quaternion
polynomials. In particular, the symmetric bilinear form
\begin{equation*}
  q\colon \SPQ \times \SPQ \to \R,
  \quad
  (h,g) \mapsto \frac{1}{2}(h \Cj{g} + g \Cj{h})
\end{equation*}
will play a vital role. Since it is of signature $(2,2)$, the real
four-dimensional vector-space $\SPQ$ together with $q$ is a
\emph{pseudo-Euclidean space.} Its \emph{null cone} consists of all split
quaternions $h$ that satisfy $q(h,h) = 0$. Because of $q(h,h) = h\Cj{h}$, these
are precisely the split quaternions of vanishing norm.

Some aspects of polynomial factorization over split quaternions have a geometric
interpretation in this pseudo-Euclidean space while others are of projective
nature. Hence, we also consider the projective space $\P(\SPQ)$ over $\SPQ$. Any
vector $h \in \SPQ \setminus \{0\}$ represents a point in $\P(\SPQ)$ which we
denote by $[h]$. Projective span is denoted by the symbol ``$\vee$'', i.e., the
straight line spanned by two different points $[h_1]$, $[h_2] \in \P(\SPQ)$ is
$[h_1] \vee [h_2]$.

\begin{defn}
  \label{def:null_quadric}
  The quadric $\N$ in $\P(\SPQ)$ represented by the symmetric bilinear form $q$
  is called the \emph{null quadric}. Lines contained in $\N$ are called
  \emph{null lines}.
\end{defn}

Because the signature of $q$ is $(2,2)$, the null quadric $\N$ is of hyperbolic
type and two families of null lines do exist. We illustrate the importance of
$\N$ to the algebra of split quaternions by a few results which we will need
later.

\begin{lem}[\cite{li18a}]
  \label{lem:null-lines}
  Let $r_0$, $r_1 \in \SPQ$ be two linearly independent split quaternions. The
  straight line $[r_0] \vee [r_1]$ is a null line if and only if the polynomial
  $R = r_1 t + r_0$ satisfies $R \Cj{R} = 0$.
\end{lem}
\begin{proof}
  The line $[r_0] \vee [r_1]$ is contained in $\N$ if and only if $q(r_0,r_0) =
  q(r_0,r_1) = q(r_1,r_1) = 0$ \cite[Lemma~6.3.3]{casas-alvero14}. On the other
  hand we have
  \begin{align*}
    R \Cj{R} = r_1 \Cj{r}_1 t^2 + (r_0 \Cj{r}_1 + r_1 \Cj{r}_0) t + r_0 \Cj{r}_0 =
    q(r_1,r_1) t^2 + 2 q(r_0, r_1) t + q(r_0, r_0)
  \end{align*}
  and the two statements are equivalent.
\end{proof}

The two families of null lines can be distinguished by algebraic properties of
split quaternions.

\begin{thm}
  \label{th:left-right-ruling}
  If $[h]$ is a point of $\N$, the sets
  \begin{equation}
    \label{eq:left-right-ruling}
    \mathcal{L} \coloneqq \{ [r] \mid r\Cj{h} = 0 \}
    \quad\text{and}\quad
    \mathcal{R} \coloneqq \{ [r] \mid \Cj{h}r = 0 \}
  \end{equation}
  are the two different rulings of $\N$ through~$[h]$.
\end{thm}

\begin{proof}
  We observe that the system of homogeneous linear equations in the coefficients
  of $x \in \SPQ$ resulting from $xg = 0$ (or $gx = 0$) with $g \in \SPQ
  \setminus \{0\}$ has non trivial solutions if and only if $[g] \in \N$. In
  this case, the vector-space of solutions is of dimension two. This already
  implies that $\mathcal{L}$ and $\mathcal{R}$ are straight lines.

  Consider $[r] \in \mathcal{L}$. We have $2 \, q(h,r) = h\Cj{r} + r\Cj{h} =
  \Cj{(r\Cj{h})} + r\Cj{h} = 0$ and $\mathcal{L}$ lies in the tangent plane of
  $\N$ in $[h]$. We choose $p \in \SPQ$ such that $ph \neq 0$ and $[ph] \neq
  [h]$. This is possible since the solutions set of $xh = 0$ is a vector-space of
  dimension two and the set of all real multiples of $h$ is a vector-space of
  dimension one. From $(ph)\Cj{(ph)} = p(h\Cj{h})\Cj{p} = 0$ we infer that the
  point $[ph]$ lies on $\N$. Moreover, it is contained in $\mathcal{L}$ because
  of $(ph)\Cj{h} = p(h\Cj{h}) = 0$. Obviously, $h$ is contained in $\mathcal{L}$
  as well. Summing up, $\mathcal{L}$ is in the tangent plane of $\N$ in $[h]$
  and contains two distinct points $[h]$, $[ph] \in \N$. Hence, it is a null
  line through~$h$.

  Similar arguments demonstrate that $\mathcal{R}$ is a null line through $[h]$
  as well. The two rulings $\mathcal{L}$ and $\mathcal{R}$ cannot be the same as
  the equation systems $hx = 0$ and $xh = 0$ are not equivalent.
\end{proof}

\begin{defn}
  \label{def:left-right-ruling}
  The ruling $\mathcal{L}$ in Theorem~\ref{th:left-right-ruling} is called a
  \emph{left ruling} of $\N$, the ruling $\mathcal{R}$ is called a \emph{right
    ruling.}
\end{defn}

\begin{cor}
  \label{cor:left-right-ruling2}
  Consider two points $[h] \in \N$ and $[p] \in \P(\SPQ)$.
  \begin{itemize}
  \item If $p$ is such that $ph \neq 0$ and $[ph] \neq [h]$, then $[h] \vee
    [ph]$ is a left ruling of~$\N$.
  \item If $p$ is such that $hp \neq 0$ and $[hp] \neq [h]$, then $[h] \vee
    [hp]$ is a right ruling of~$\N$.
  \end{itemize}
\end{cor}

\begin{proof}
  Consider $p \in \SPQ$ with $ph \neq 0$ and $[ph] \neq [h]$. As argued in the
  proof of Theorem~\ref{th:left-right-ruling}, such a choice is possible. We
  have $h\Cj{h} = 0$ and $ph\Cj{h} = 0$. This implies that $[h]$ and $[ph]$ lie
  on the same left ruling by Theorem~\ref{th:left-right-ruling}. Therefore we
  have two different points on a left ruling and the span $[h] \vee [ph]$ of
  these points is the ruling itself. The second statement is similar.
\end{proof}

\begin{rmk}
  For fixed $p \in \SPQ \setminus \{0\}$ the maps $[x] \mapsto [px]$ and $[x]
  \mapsto [xp]$ are the well-known Clifford left and right translations of
  non-Euclidean geometry.
\end{rmk}

\begin{cor}
  \label{cor:affine-two-plane-of-zeros}
  Given split quaternions $h$, $g \in \SPQ \setminus \{0\}$ such that $[h] \vee
  [g]$ is a left (right) ruling of $\N$, there exists an affine two-plane of
  split quaternions $p$ such that $g = ph$ ($g = hp$).
\end{cor}
\begin{proof}
  The split quaternion equation $g = xh$ results in a system of in-homogeneous
  linear equations for the coefficients of $x$. We already argued in our proof
  of Theorem~\ref{th:left-right-ruling} that the solution space of the
  corresponding system of homogeneous equations is of dimension two.
\end{proof}

Corollary~\ref{cor:affine-two-plane-of-zeros} is a pure existence result. The
next theorem provides a parametrization of the affine two-plane in
Corollary~\ref{cor:affine-two-plane-of-zeros}. The main idea of the proof is to
derive properties of a split quaternion $p = p_0 + p_1 \qi + p_2 \qj + p_3 \qk$
by finding relations between its ``positive'' part $p_0 + p_1 \qi$ and its
``negative'' part $p_2 \qj + p_3 \qk$. These terms are motivated by the sign of
their respective norms.

\begin{thm}
  \label{th:affine-two-plane-of-zeros}
  Suppose that $h = h_0 + h_1 \qi + h_2 \qj + h_3 \qk \in \SPQ \setminus \{ 0
  \}$ and $g = g_0 + g_1 \qi + g_2 \qj + g_3 \qk \in \SPQ \setminus \{ 0 \}$ are
  as in Corollary~\ref{cor:affine-two-plane-of-zeros}. The affine two-plane
  consisting of all split quaternions $x \in \SPQ$ solving the equation $g = xh$
  can be parameterized by $u + \lambda \Cj{h} + \mu \qi \Cj{h}$, where $u = (g_0
  + g_1 \qi) (h_0 + h_1 \qi)^{-1}$ and $\lambda$, $\mu \in \R$. (The same
  statement holds for $g = hx$ with $u + \lambda \Cj{h} + \mu \Cj{h} \qi$ where
  $u = (h_0 + h_1 \qi)^{-1} (g_0 + g_1 \qi)$ and $\lambda$, $\mu \in \R$.)
\end{thm}
\begin{proof}
  Regarding the system of linear equations arising from $xh = g$ we have to show
  the following:
  \begin{itemize}
  \item $u$ solves $xh = g$,
  \item $\Cj{h}$ and $\qi\Cj{h}$ solve the corresponding homogeneous system $xh
    = 0$, and
  \item $\Cj{h}$ and $\qi\Cj{h}$ are linearly independent.
  \end{itemize}
  (Note that we already know that the solution space is of dimension two.)

  The condition that $[h] \vee [g]$ is a null line yields
  \begin{equation}
    \label{eq:affine-two-plane-of-zeros2}
    h \Cj{h} = g \Cj{g} = h \Cj{g} + g \Cj{h} = 0.
  \end{equation}
  Moreover, it is a left ruling of $\N$ and therefore
  \begin{equation}
    \label{eq:affine-two-plane-of-zeros3}
    h \Cj{g} = g \Cj{h} = 0
  \end{equation}
  by Equation~\eqref{eq:left-right-ruling}. From
  Equation~\eqref{eq:affine-two-plane-of-zeros2} we obtain
  \begin{equation}
    \label{eq:affine-two-plane-of-zeros4}
    0 = h \Cj{h} = (h_0 + h_1 \qi) \Cj{(h_0 + h_1 \qi)} + (h_2 \qj + h_3 \qk) \Cj{(h_2 \qj + h_3 \qk)}
  \end{equation}
  and, because $h \neq 0$, the norms of $h_0 + h_1 \qi$ and $h_2 \qj + h_3 \qk$
  are different from zero. Hence, $h_0 + h_1 \qi$ and $h_2 \qj + h_3 \qk$ are
  both invertible and $u = (g_0 + g_1 \qi) (h_0 + h_1 \qi)^{-1}$ is well
  defined. We have
  \begin{equation}
  \begin{aligned}
    \label{eq:affine-two-plane-of-zeros5}
    u h &= (g_0 + g_1 \qi) (h_0 + h_1 \qi)^{-1} (h_0 + h_1 \qi + h_2 \qj + h_3 \qk) \\
    &= g_0 + g_1 \qi + (g_0 + g_1 \qi) (h_0 + h_1 \qi)^{-1} (h_2 \qj + h_3 \qk) \\
    &= g_0 + g_1 \qi + \frac{1}{(h_0 + h_1 \qi) \Cj{(h_0 + h_1 \qi)}} (g_0 + g_1 \qi) \Cj{(h_0 + h_1 \qi)} (h_2 \qj + h_3 \qk).
  \end{aligned}
  \end{equation}
  Equation~\eqref{eq:affine-two-plane-of-zeros3} yields
  \begin{align*}
    0 = g \Cj{h} = \ &(g_0 + g_1 \qi) \Cj{(h_0 + h_1 \qi)} + (g_2 \qj + g_3 \qk) \Cj{(h_2 \qj + h_3 \qk)} + \\ &(g_0 + g_1 \qi) \Cj{(h_2 \qj + h_3 \qk)} + (g_2 \qj + g_3 \qk) \Cj{(h_0 + h_1 \qi)}
  \end{align*}
  where the first two terms form the positive part (they are in the span of $1$
  and $\qi$) while the two trailing terms form the negative part (they are in
  the span of $\qj$ and $\qk$). Positive and negative parts both have to vanish whence
  \begin{equation}
  \begin{aligned}
    \label{eq:affine-two-plane-of-zeros6}
    0 = \ &(g_0 + g_1 \qi) \Cj{(h_0 + h_1 \qi)} + (g_2 \qj + g_3 \qk) \Cj{(h_2 \qj + h_3 \qk)} \\
    \Leftrightarrow \, &(g_0 + g_1 \qi) \Cj{(h_0 + h_1 \qi)} (h_2 \qj + h_3 \qk) = -(g_2 \qj + g_3 \qk) \Cj{(h_2 \qj + h_3 \qk)} (h_2 \qj + h_3 \qk)
  \end{aligned}
  \end{equation}
  Substituting Equation~\eqref{eq:affine-two-plane-of-zeros6} into
  \eqref{eq:affine-two-plane-of-zeros5} we obtain via
  \eqref{eq:affine-two-plane-of-zeros4}
  \begin{equation*}
    u h = g_0 + g_1 \qi - \frac{(h_2 \qj + h_3 \qk) \Cj{(h_2 \qj + h_3 \qk)}} {(h_0 + h_1 \qi) \Cj{(h_0 + h_1 \qi)}} (g_2 \qj + g_3 \qk) = g_0 + g_1 \qi + g_2 \qj + g_3 \qk = g
  \end{equation*}
  and $u$, indeed, solves $xh = g$.
  
  The split quaternions $\Cj{h} = h_0 - h_1\qi - h_2\qj - h_3\qk$ and $\qi\Cj{h}
  = h_1 + h_0\qi + h_3\qj - h_2\qk$ obviously solve the homogeneous system $xh =
  0$. It remains to be shown that they are linearly independent: The positive
  parts are linearly dependent if and only if $h_0 = h_1 = 0$, the negative
  parts are linearly dependent if and only if $h_2 = h_3 = 0$. Both conditions
  cannot be fulfilled as $h \neq 0$.
\end{proof}

Note that the statements of Corollary~\ref{cor:affine-two-plane-of-zeros} and
Theorem~\ref{th:affine-two-plane-of-zeros} hold true even if $[g]$ and $[h]$ do
not span a line but coincide.

\section{Factorization Results}
\label{sec:factorization-results}

In this section, we investigate factorizability of quadratic split quaternion
polynomials. Consider a quadratic polynomial
\begin{equation}
  \label{eq:polynomial}
  P = a t^2 + b t + c \in \SPQ[t],
\end{equation}
where $a = a_0 + a_1 \qi + a_2 \qj + a_3 \qk$, $b = b_0 + b_1 \qi + b_2 \qj +
b_3 \qk$, $c = c_0 + c_1 \qi + c_2 \qj + c_3 \qk$ are split quaternions. We say
that $P$ \emph{admits a factorization,} if there exist split quaternions $h_1$,
$h_2$ such that
\begin{equation*}
  P = a(t - h_1)(t - h_2).
\end{equation*}
For the time being (until Section~\ref{sec:factor-non-invertible}) we assume
that the leading coefficient $a$ is invertible. In this case, we may further
assume that $P$ is monic because we may easily construct all factorization of
$P$ from factorizations of the monic polynomial $a^{-1}P$. Finally, we apply the
parameter transformation $t \mapsto t - \frac{b_0}{2}$ whence $b_0 = 0$. To
summarize, we investigate the factorizations
\begin{equation*}
  P = t^2 + bt + c = (t-h_1)(t-h_2)
\end{equation*}
where $b$, $c$, $h_1$, $h_2 \in \SPQ$ and $\RE b = 0$ (or, equivalently, $b +
\Cj{b} = 0$).

A fundamental result (for example \cite[Theorem~2]{li18d}) relates
factorizations to right zeros:
\begin{lem}
  \label{lem:zero-factor}
  The split quaternion $h_2$ is a right zero of the (not necessarily monic)
  polynomial $P \in \SPQ[t]$ if and only if $t - h_2$ is a right factor of~$P$.
\end{lem}
Once a right factor $t - h_2$ of a quadratic polynomial is found, a left factor
$t - h_1$ can be computed by left polynomial division. Thus finding
factorizations is essentially equivalent to finding right zeros and all results
on right zeros of \cite{abrate09,cao19} are of relevance to us. Nonetheless, we
continue by developing our own criteria that are related to a well-known
procedure \cite{hegedus13,li18d} for computing a factorization of a generic
quadratic polynomial~$P$:
\begin{itemize}
  \label{alg:factorization}
\item Pick a monic quadratic factor $M \in \R[t]$ of the norm polynomial
  $P\Cj{P}$.
\item Compute the remainder polynomial $R$ of $P$ when dividing by $M$. Since
  $P$ and $M$ are monic we have $P = M + R$. Moreover, $\deg R \le 1$.
\item If $R \Cj{R} \neq 0$, then $R$ has a unique zero $h_2 \in \SPQ$. The
  linear split quaternion polynomial $t - h_2$ is not only a factor of $R$, but
  also of $M$ and therefore a factor of~$P$.
\item Right division of $P$ by $t - h_2$ yields the factorization $P = (t - h_1)
  (t - h_2)$ with $h_1$, $h_2 \in \SPQ$.
\end{itemize}
We refer to above construction as \emph{generic factorization algorithm}. It is
sufficient unless $R \Cj{R} = 0$. In this case the remainder polynomial $R$
might not have a zero at all. If it has a zero, it already has infinitely many
zeros but it is not guaranteed that they lead to right factors. In this sense,
factorization of split quaternion polynomials is more interesting than
factorization of polynomials over the division ring of ordinary (Hamiltonian)
quaternions.

The goal of this section is to provide necessary and sufficient criteria for
factorizability of all monic quadratic split quaternion polynomials $P = t^2 +
bt + c$. In doing so, we consider the following sub-cases:
\begin{itemize}
\item $b$, $c \in \R$ $\leadsto$
  Corollary~\ref{cor:factor-polynomial-b-real-c-real},
  Lemma~\ref{lem:factor-polynomial-b-real-c-real}
\item $b \in \R$ and $c \in \SPQ$ $\leadsto$
  Theorem~\ref{th:factor-polynomial-b-real-c-nonreal}
\item $b$, $c \in \SPQ$ and $1$, $b$, $c$ linearly dependent $\leadsto$
  Theorem~\ref{th:factor-polynomial-b-nonreal-c-nonreal}
\item $b$, $c \in \SPQ$ and $1$, $b$, $c$ linearly independent $\leadsto$
  Theorem~\ref{th:factor-polynomial-monic-independent-coefficients}
\end{itemize}

Our structure differs from \cite{abrate09} but we can draw direct connections to
some results. Lemma~\ref{lem:factor-polynomial-b-real-c-real} is related to
\cite[Theorem~2.4.1]{abrate09}, \cite[Theorem~2.2]{abrate09} provides a formula
to compute the roots of a polynomial given that the linear coefficient is not
invertible. Although the condition for the formula to yield all roots is not
met, it can be used to obtain our result in
Theorem~\ref{th:factor-polynomial-b-real-c-nonreal}. The combination of
\cite[Theorem~2.2]{abrate09} and \cite[Theorem~2.4.2]{abrate09} covers the
second item in our Theorem~\ref{th:factor-polynomial-b-nonreal-c-nonreal}.

Similar to the structure in \cite{cao19} we begin our discussion with the case
that the linear coefficient of $P$ is real. The combination of our
Lemma~\ref{lem:factor-polynomial-b-real-c-real} and
Theorem~\ref{th:factor-polynomial-b-real-c-nonreal} is equivalent to
\cite[Theorem~2.1 and Theorem~2.2]{cao19}. Finally,
Theorem~\ref{th:factor-polynomial-b-nonreal-c-nonreal} and
Theorem~\ref{th:factor-polynomial-monic-independent-coefficients} cover the
statements in \cite[Theorem~4.1 and Theorem~4.2]{cao19}.

Regarding the remaining results in \cite{abrate09} or \cite{cao19} it is not so
straightforward to draw direct connections, one would find a set of polynomials
which need to be treated by different cases with respect to our characterization
but can be covered by only one theorem in \cite{abrate09} or \cite{cao19} and
vice versa.

A split quaternion $x = x_0 + x_1 \qi + x_2 \qj + x_3 \qk \in \SPQ$ is a zero of
$P = t^2 + bt + c$ with $\Scalar{b} = 0$ if and only if it solves the real
system of nonlinear equations
\begin{equation}
  \label{eq:equation-system}
  \begin{aligned}
    x_0^2 - x_1^2 + x_2^2 + x_3^2 - b_1 x_1 + b_2 x_2 + b_3 x_3 + c_0 &= 0, \\
    2 x_0 x_1 + b_1 x_0 + b_3 x_2 - b_2 x_3 + c_1 &= 0, \\
    2 x_0 x_2 + b_2 x_0 + b_3 x_1 - b_1 x_3 + c_2 &= 0, \\
    2 x_0 x_3 + b_3 x_0 - b_2 x_1 + b_1 x_2 + c_3 &= 0.
  \end{aligned}
\end{equation}
In view of Lemma~\ref{lem:zero-factor}, it gives rise to a right factor $t - x$
of~$P$. Above system is obtained by evaluating $P$ at $x$ and equating the
coefficients of the quaternion units $\qi$, $\qj$, $\qk$ and the real
coefficient with zero. Note that we are only interested in real solutions. A
priori it is not obvious that this system has a solution at all. Indeed, there
exist examples with zero as well as with infinitely many solutions. Below we
present necessary and sufficient conditions for solvability in all cases along
with some solutions.

\subsection{Factorization of Monic Polynomials with Dependent Coefficients}
\label{sec:factor-monic-dependent}

To begin with, we determine the zeros of the polynomial $P$ in
\eqref{eq:polynomial} supposing that $P$ is real. In addition to the general
assumptions $a = 1$ and $b_0 = 0$ this means $b_1 = b_2 = b_3 = c_1 = c_2 = c_3
= 0$. The factorization algorithm for generic polynomials (described on
Page~\pageref{alg:factorization}) fails in this setup. However, we can directly
solve the polynomial system \eqref{eq:equation-system}.

\begin{lem}
  \label{lem:factor-polynomial-b-real-c-real}
  The polynomial $P = t^2 + c_0 \in \SPQ[t]$, where $c_0 \in \R$, has infinitely
  many split quaternion zeros given by the set $\{ x \in \SPQ : x_0 = 0, x\Cj{x}
  = c_0 \}$. In addition, if $c_0 \le 0$, there are two real zeros $x = \pm
  \sqrt{- c_0}$ which coincide for $c_0 = 0$.
\end{lem}

\begin{proof}
  We solve the equation system \eqref{eq:equation-system} for $x = x_0 + x_1 \qi
  + x_2 \qj + x_3 \qk$ under the additional assumption that $b_1 = b_2 = b_3 =
  c_1 = c_2 = c_3 = 0$. If $x_0 \neq 0$, we have $x_1 = x_2 = x_3 = 0$ and
  $x_0^2 + c_0 = 0$. Hence, $x$ is real and a real solution exists if and only
  if $c_0 \le 0$. If so, there are two (possibly identical) solutions $x =
  \pm\sqrt{- c_0}$. If $x_0 = 0$, the system simplifies to the single equation
  $- x_1^2 + x_2^2 + x_3^2 + c_0 = 0$. Since $- x_1^2 + x_2^2 + x_3^2 = - x
  \Cj{x}$, provided that $x_0 = 0$, the set of solutions reads as $\{ x \in
  \SPQ\colon x_0 = 0, x\Cj{x} = c_0 \}$. It is always infinite.
\end{proof}

Combining Lemma~\ref{lem:factor-polynomial-b-real-c-real} with
Lemma~\ref{lem:zero-factor}, we can state

\begin{cor}
  \label{cor:factor-polynomial-b-real-c-real}
  The polynomial $P = t^2 + c_0 \in \SPQ[t]$ with $c_0 \in \R$ admits infinitely
  many factorizations over~$\SPQ$.
\end{cor}

The solution set $\{ x \in \SPQ\colon x_0 = 0, x\Cj{x} = c_0 \}$ defines a
hyperboloid of one sheet, a cone or a hyperboloid of two sheets for $c_0 < 0,
c_0 = 0$ or $c_0 > 0$, respectively, in the affine space $\Vector{\SPQ}$.

\begin{rmk}
  Similar results can be obtained for Hamiltonian quaternions \cite{huang02}.
  One noteworthy difference is non-negativity of the Hamiltonian norm $x_0^2 +
  x_1^2 + x_2^2 + x_3^2$, whence, $x_1^2 + x_2^2 + x_3^2 = c_0$ has no solution
  for $c_0 < 0$ and $P$ has just two real zeros $\pm\sqrt{-c_0}$. The zero set
  of $P \in \R[t]$ over the split quaternions is always infinite.
\end{rmk}

We continue by considering monic polynomials whose constant coefficient is not
real. Such a polynomial is given by $P$ in \eqref{eq:polynomial} with $a = 1$,
$b_1 = b_2 = b_3 = 0$ and $c \notin \R$.

\begin{thm}
  \label{th:factor-polynomial-b-real-c-nonreal}
  The polynomial $P = t^2 + c \in \SPQ[t]$, where $c \in \SPQ \setminus \R$,
  admits a factorization if and only if
  \begin{itemize}
  \item $\Vector{c} \ \Cj{\Vector{c}} > 0$ or
  \item $c \Cj{c} \geq 0$ and $c_0 < 0$.
  \end{itemize}
\end{thm}

\begin{proof}
  We solve the equation system \eqref{eq:equation-system} which, in the current
  setup, reads as
  \begin{align*}
    x_0^2 - x_1^2 + x_2^2 + x_3^2 + c_0 &= 0, \\
    2 x_0 x_1 + c_1 &= 0, \\
    2 x_0 x_2 + c_2 &= 0, \\
    2 x_0 x_3 + c_3 &= 0.
  \end{align*}
  The assumption $x_0 = 0$ implies $c_1 = c_2 = c_3 = 0$ and contradicts $c
  \notin \R$. Hence, we can plug $ x_1 = -\frac{c_1}{2 x_0}$, $x_2 =
  -\frac{c_2}{2 x_0}$, $x_3 = -\frac{c_3}{2 x_0}$, in the first equation and
  obtain $\frac{1}{4 x_0^2} (4 x_0^4 + 4 c_0 x_0^2 - c_1^2 + c_2^2 + c_3^2) = 0$
  with up to four distinct solutions
  \begin{equation}
    \label{eq:zeros-polynomial-b-real-c-nonreal}
    \begin{aligned}
      x_0 = \pm \frac{1}{\sqrt{2}}\sqrt{- c_0 \pm \sqrt{c_0^2 + c_1^2 - c_2^2 - c_3^2}} &= \pm \frac{1}{\sqrt{2}}\sqrt{- c_0 \pm \sqrt{c \Cj{c}}} \\
      &= \pm \frac{1}{\sqrt{2}}\sqrt{- c_0 \pm \sqrt{c_0^2 + \Vector{c}\Cj{\Vector{c}}}}
    \end{aligned}
  \end{equation}
  over $\C$. We are only interested in real solutions and it is easy to see that
  all expressions in \eqref{eq:zeros-polynomial-b-real-c-nonreal} yield non-real
  values if and only if $c \Cj{c} < 0$ or $c \Cj{c} \ge 0$, $c_0 > 0$ and
  $\Vector{c} \ \Cj{\Vector{c}} < 0$. We already verified $x_0 = 0$ to be
  invalid whence also the case $\Vector{c} \ \Cj{\Vector{c}} = 0$ and $c_0 \ge
  0$ is excluded.

  Hence, $P$ has no split quaternion zeros and therefore does not admit a
  factorization if and only if
  \begin{itemize}
  \item $c \Cj{c} < 0$ or
  \item $c \Cj{c} \ge 0$, $c_0 > 0$ and $\Vector{c} \ \Cj{\Vector{c}} < 0$ or
  \item $\Vector{c} \ \Cj{\Vector{c}} = 0$ and $c_0 \ge 0$.
  \end{itemize}
  The negation of these conditions are easily shown to be equivalent to
  \begin{itemize}
  \item $\Vector{c} \ \Cj{\Vector{c}} > 0$ or
  \item $c \Cj{c} \geq 0$ and $c_0 < 0$,
  \end{itemize}
  thus finishing the proof.
\end{proof}

Still assuming that $P$ is monic, we are left with the case where $b \notin \R$.
Due to the assumed dependency of the coefficients there exist $\lambda$, $\mu
\in \R$ such that $c = \lambda + \mu b$ and we can write $P = t^2 + b t +
\lambda + \mu b$.

\begin{thm}
  \label{th:factor-polynomial-b-nonreal-c-nonreal}
  Consider the split quaternion polynomial $P = t^2 + b t + \lambda + \mu b \in
  \SPQ[t]$, where $b \in \SPQ \setminus \R$, $b_0 = 0$ and $\lambda$, $\mu \in
  \R$.
  \begin{itemize}
  \item If $b \Cj{b} > 0$, then $P$
    admits a factorization.
  \item Provided that $b \Cj{b} = 0$, then $P$ admits a factorization if and
    only if $\lambda + \mu^2 = 0$ or $\lambda < 0$.
  \item Provided that $b \Cj{b} < 0$, then $P$ admits a factorization if and
    only if $\lambda + \mu^2 = 0$ or $b \Cj{b} + 4 \lambda \le 0$ and $b \Cj{b}
    + 4 \lambda \le 4 \mu \sqrt{- b \Cj{b}} \le -(b \Cj{b} + 4 \lambda)$.
  \end{itemize}
\end{thm}

\begin{proof}
  First, let us assume that $b \Cj{b} > 0$. We pick a quadratic factor $M = t^2
  + m_1 t + m_0 \in \R[t]$ of the norm polynomial $P \Cj{P}$ and compute the
  remainder polynomial $R = P - M = (b - m_1) t + \lambda + \mu b - m_0$ when
  dividing $P$ by $M$. By the generic factorization algorithm, $P$ admits a
  factorization if the leading coefficient of $R$ is invertible. This is
  guaranteed by non-negativity of its norm $(b - m_1) \Cj{(b - m_1)} = b \Cj{b}
  + m_1^2 > 0$.

  Next, we assume that $b \Cj{b} = 0$. If $\lambda + \mu^2 = 0$, then obviously
  $P = (t + \mu) (t - \mu + b)$ is a factorization. If $\lambda < 0$, then $P$
  factors as $P = (t - h_1) (t - h_2)$ where
  \begin{align*}
    h_{1,2} = \pm \sqrt{-\lambda} - \frac{1}{2} b (1 \pm \frac{\mu}{\sqrt{-\lambda}}).
  \end{align*}
  Conversely, if $P$ admits a factorization, then $P$ has a right zero. Such a
  zero is a solution of the equation system \eqref{eq:equation-system} which, in
  our case, reads as
  \begin{equation}
    \label{eq:equation-system-dependent}
    \begin{aligned}
      x_0^2 - x_1^2 + x_2^2 + x_3^2 - b_1 x_1 + b_2 x_2 + b_3 x_3 + \lambda &= 0, \\
      2 x_0 x_1 + b_1 x_0 + b_3 x_2 - b_2 x_3 + \mu b_1 &= 0, \\
      2 x_0 x_2 + b_2 x_0 + b_3 x_1 - b_1 x_3 + \mu b_2 &= 0, \\
      2 x_0 x_3 + b_3 x_0 - b_2 x_1 + b_1 x_2 + \mu b_3 &= 0.
    \end{aligned}
  \end{equation}
  Assuming that $x_0 \neq 0$, we can substitute $x_1 = -(b_1 (\mu + x_0))(2
  x_0)^{-1}$, $x_2 = -(b_2 (\mu + x_0))(2 x_0)^{-1}$, $x_3 = -(b_3 (\mu +
  x_0))(2 x_0)^{-1}$ into the first equation and obtain $(4x_0^4 + (b \Cj{b} +
  4\lambda) x_0^2 - \mu^2 b \Cj{b})(4x_0^2)^{-1} = x_0^2 + \lambda = 0$. Since
  $x_0 \neq 0$, a real solution exists if and only if $\lambda < 0$. Considering
  solutions with $x_0 = 0$ the equations system
  \eqref{eq:equation-system-dependent} simplifies to
  \begin{align*}
      -x_1^2 + x_2^2 + x_3^2 - b_1 x_1 + b_2 x_2 + b_3 x_3 + \lambda &= 0, \\
      b_3x_2 - b_2 x_3 + \mu b_1 &= 0, \\
      b_3x_1 - b_1 x_3 + \mu b_2 &= 0, \\
      -b_2x_1 + b_1 x_2 + \mu b_3 &= 0.
  \end{align*}
  The conditions $b \notin \R$ and $b \Cj{b} = 0$ imply $b_1 \neq 0$ and the
  solution set of the equation system given by the last three equations is of
  dimension one. It can be parameterized by $\{ x_1 = \alpha,\ x_2 = (b_2 \alpha -
  b_3 \mu)b_1^{-1},\ x_3 = (b_3 \alpha + b_2 \mu)b_1^{-1}\colon \alpha \in \R
  \}$. Substituting these solutions into the first equation yields $\lambda +
  \mu^2 = 0$. This concludes the proof of the second statement.

  Assuming that $b \Cj{b} < 0$, we can factor the norm polynomial as
  \begin{multline*}
    P \Cj{P} =
    (t^2 + \lambda)^2 + (t + \mu)^2 b \Cj{b}
    = (t^2 + \lambda + \sqrt{- b \Cj{b}} (t + \mu)) (t^2 + \lambda - \sqrt{- b \Cj{b}} (t + \mu)).
  \end{multline*}
  If $b \Cj{b} + 4 \lambda \le 0$ and $b \Cj{b} + 4 \lambda \le 4 \mu \sqrt{- b
    \Cj{b}} \le -(b \Cj{b} + 4 \lambda)$, then $P \Cj{P}$ has even real linear
  factors
  \begin{align*}
    P \Cj{P} &= L_1 L_2 L_3 L_4, \quad \text{where} \\
    L_{1,2} &= t + \frac{1}{2} \left( \sqrt{- b \Cj{b}} \pm \sqrt{- (b \Cj{b} + 4 \lambda) - 4 \sqrt{- b \Cj{b}} \mu} \right), \\
    L_{3,4} &= t - \frac{1}{2} \left( \sqrt{- b \Cj{b}} \pm \sqrt{- (b \Cj{b} + 4 \lambda) + 4 \sqrt{- b \Cj{b}} \mu} \right).
  \end{align*}
  Defining $M \coloneqq L_1 L_4$ and computing $R = r_1 t + r_0 = P - M \in
  \SPQ[t]$ yields a remainder polynomial with leading coefficient
  \begin{align*}
    r_1 = b -
    \frac{1}{2} \left( \sqrt{- (b \Cj{b} + 4 \lambda) + 4 \sqrt{- b \Cj{b}} \mu} +
    \sqrt{- (b \Cj{b} + 4 \lambda) - 4 \sqrt{- b \Cj{b}} \mu} \right).
  \end{align*}
  The polynomial $P$ admits a factorization by means of the generic
  factorization algorithm if the norm $r_1 \Cj{r_1} = b \Cj{b} + \frac{1}{2} (
  \sqrt{(b \Cj{b} + 4 \lambda)^2 + 16 \mu^2 b \Cj{b}} - (b \Cj{b} + 4\lambda) )$
  of $r_1$ is different from zero. This is, indeed, the case as $b \Cj{b} \neq
  0$ and $\mu^2+\lambda \neq 0$. If $\mu^2+\lambda = 0$, then $P$ admits the
  factorization $P = (t + \mu) (t - \mu + b)$ anyway. Similar to above
  considerations, a detailed inspection of the equation system
  \eqref{eq:equation-system-dependent} shows that no solutions exist if the
  conditions $\lambda + \mu^2 = 0$ or $b \Cj{b} + 4 \lambda \le 0$ and $b \Cj{b}
  + 4 \lambda \le 4 \mu \sqrt{- b \Cj{b}} \le -(b \Cj{b} + 4 \lambda)$ are
  violated.
\end{proof}

\subsection{Factorization of Monic Polynomials with Independent Coefficients}
\label{sec:factor-monic-independent}

In \cite{li18d} the authors showed that the polynomial $P$ in
\eqref{eq:polynomial} admits a factorization if its coefficients $a$, $b$, $c$
are linearly independent and the leading coefficient $a$ is invertible.
Assuming, without loss of generality, that $a = 1$, we recall this result and
provide an improved version of the second half of the proof in \cite{li18d},
namely the case where the general factorization algorithm is not applicable.

\begin{thm}[\cite{li18b}]
  \label{th:factor-polynomial-monic-independent-coefficients}
  The polynomial $P = t^2 + b t + c \in \SPQ[t]$ admits a factorization if its
  coefficients $1$, $b$, $c$ are linearly independent.
\end{thm}
\begin{proof}
  Let $M_1 \in \R[t]$ be a monic quadratic factor of the norm polynomial $P
  \Cj{P}$ and compute the corresponding linear remainder polynomial $R_1 \in
  \SPQ[t]$ such that $P = M_1 + R_1$. If $R_1 \Cj{R_1} \neq 0$, one can compute
  a factorization of $P$ using the generic factorization algorithm. Hence, we
  continue by assuming that $R_1 \Cj{R_1} = 0$. Linear independence of the
  coefficients of $P$ implies linear independence of the coefficients of $R_1$.
  Consequently, $R_1$ parameterizes a null line (Lemma~\ref{lem:null-lines}).
  Consider the complementary monic quadratic factor $M_2 \in \R[t]$ of $P\Cj{P}$
  defined by $P \Cj{P} = M_1 M_2$, and the corresponding remainder polynomial
  $R_2$ such that $P = M_2 + R_2$. From
  \begin{equation*}
    \begin{aligned}
      P \Cj{P} = (M_1 + R_1) \Cj{(M_1 + R_1)} = M_1^2 + M_1 R_1 + M_1 \Cj{R_1} = M_1 (M_1 + R_1 + \Cj{R_1})
    \end{aligned}
  \end{equation*}
  we conclude that $M_2 = M_1 + R_1 + \Cj{R_1}$ and $R_2 = -\Cj{R_1}$. Hence,
  $R_2$ parameterizes a null line as well. The two null lines belong to different
  families of rulings of $\N$. Without loss of generality we assume that the
  null line parameterized by $R_1$ is a right ruling. Moreover, we can assume
  that the linear coefficient of $M_1$ is zero by applying a suitable parameter
  transformation ($t \mapsto t + \tilde m$ where $\tilde m \in \R$) to~$P$.

  Next we will show that $M_1 = t^2 + m \in \R[t]$ and $R_1 = r_1 t + r_0$ have
  a common right zero. By Corollary~\ref{cor:affine-two-plane-of-zeros}, there
  exists an $h \in \SPQ$ such that $-r_0 = r_1 h$. Although
  Theorem~\ref{th:affine-two-plane-of-zeros} provides an explicit formula to
  compute such an $h \in \SPQ$ in terms of $r_1$ and $r_0$, any $h \in \SPQ$
  fulfilling the relation $-r_0 = r_1 h$ will do and we choose one. Then the
  two-parametric set of right zeros of $R_1$ can be parameterized by $h + \lambda
  \Cj{r_1} + \mu \Cj{r_1} \qi$ where $\lambda$, $\mu \in \R$. The norm of such
  an element reads as
  \begin{align*}
    &\phantom{=}\ \Cj{(h + \lambda \Cj{r_1} + \mu \Cj{r_1} \qi)} (h + \lambda \Cj{r_1} + \mu \Cj{r_1} \qi) \\
    &= (\Cj{h} + \lambda r_1 - \mu \qi r_1) (h + \lambda \Cj{r_1} + \mu \Cj{r_1} \qi) \\
    &= \Cj{h} h + \lambda \Cj{h} \Cj{r_1} + \mu \Cj{h} \Cj{r_1} \qi + \lambda r_1 h - \mu \qi r_1 h\\
    &= \Cj{h} h - \lambda \Cj{r_0} - \mu \Cj{r_0} \qi - \lambda r_0 + \mu \qi r_0 \\
    &= \Cj{h} h - \lambda (\Cj{r_0} + r_0) - \mu (\Cj{r_0} \qi - \qi r_0).
  \end{align*}
  We choose $\lambda$ and $\mu$ such that this norm is equal to $m$, the
  constant coefficient of $M_1$, and in addition the real part of $h + \lambda
  \Cj{r_1} + \mu \Cj{r_1} \qi$ is equal to zero. This is possible because the
  coefficient matrix
  \begin{equation*}
    \begin{pmatrix}
      -r_0 - \Cj{r_0} & -\Cj{r_0}\qi + \qi r_0 \\
      \Scalar{\Cj{r_1}} & \Scalar{\Cj{r_1}\qi}
    \end{pmatrix}
  \end{equation*}
  of the underlying system of linear equations is singular precisely if the
  positive parts of $r_0$ and $r_1$ are linearly dependent. But then the null
  line spanned by $[r_0]$ and $[r_1]$ contains a point with zero positive part
  which is not possible. With above's choice we have that $h + \lambda \Cj{r_1}
  + \mu \Cj{r_1} \qi$ satisfies all conditions from
  Lemma~\ref{lem:factor-polynomial-b-real-c-real}. Hence, $h + \lambda \Cj{r_1}
  + \mu \Cj{r_1} \qi$ is not only a right zero of $R_1$ but also a zero of
  $M_1$. Therefore, it is also a zero of $P$ whence $P$ admits a factorization.
\end{proof}

The following example illustrates the ``interesting'' case in the proof of
Theorem~\ref{th:factor-polynomial-monic-independent-coefficients}.

\begin{example}
  Consider the polynomial $P = t^2 + (1 + \qk) t + 2 + \qi + \qj + \qk \in
  \SPQ$. Its norm polynomial factors into $P \Cj{P} = M_1 M_2$ with $M_1 = t^2 +
  1$ and $M_2 = t^2 + 2t + 3$. The respective remainder polynomials $R_1$, $R_2
  \in \SPQ$ such that $P = M_1 + R_1 = M_2 + R_2$ read as $R_1= (1 + \qk) t + 1
  + \qi + \qj + \qk$ and $R_2 = (\qk - 1) t - 1 + \qi + \qj + \qk$. Both, $R_1$
  and $R_2$ are null lines since $R_1 \Cj{R_1} = R_2 \Cj{R_2} = 0$, whereas only
  $R_1$ is a right ruling of $\N$. According to
  Theorem~\ref{th:affine-two-plane-of-zeros}, the two-parametric set of right
  zeros of $R_1 = r_1 t + r_0 = (1 + \qk) t + 1 + \qi + \qj + \qk$ is
  parameterized by $h + \lambda \Cj{r_1} + \mu \Cj{r_1} \qi$ with $\lambda$, $\mu
  \in \R$ and $h = - 1 - \qi \in \SPQ$. The conditions on the norm and the real
  part of $h + \lambda \Cj{r_1} + \mu \Cj{r_1} \qi$ yield the two equations
  $\lambda - 1 = 0$ and $1 - 2 \lambda - 2 \mu = 0$ with the unique solution
  $\lambda = 1$ and $\mu = -\frac{1}{2}$. Indeed, the split quaternion $h +
  \Cj{r_1} - \frac{1}{2} \Cj{r_1} \qi = -\frac{3}{2} \qi + \frac{1}{2} \qj -
  \qk$ is a zero of $P$ and right division of $P$ by $t + \frac{3}{2} \qi -
  \frac{1}{2} \qj + \qk$ yields the factorization $P = (t + 1 - \frac{3}{2} \qi
  + \frac{1}{2} \qj) (t + \frac{3}{2} \qi - \frac{1}{2} \qj + \qk)$.

  Note that the two polynomials $M_1$ and $M_2$ are irreducible, hence there is
  no quadratic factor of the norm polynomial $P \Cj{P}$ yielding a non-null line
  as remainder polynomial and therefore the possibility to avoid above's
  procedure.
\end{example}

\subsection{Geometric Interpretation for Factorizability of Monic Polynomials}
\label{sec:geometric-interpretation}

Theorem~\ref{th:factor-polynomial-b-real-c-nonreal} and
Theorem~\ref{th:factor-polynomial-b-nonreal-c-nonreal} relate factorizability of
a quadratic split quaternion polynomial to validity of certain inequalities.
Some of these conditions are not very intuitive but necessary in order to cover
all special cases by the algebraic approach. However, we can give an alternative
characterization of factorizability by interpreting the factorization algorithm
for quadratic split quaternions geometrically. It turns out that this
alternative characterization covers the statement in
Theorem~\ref{th:factor-polynomial-monic-independent-coefficients} as well.
Hence, the geometrical approach allows a unified characterization of
factorizability for quadratic split quaternions with invertible leading
coefficient without inconvenient case distinctions.

Consider a monic split quaternion polynomial $P = t^2 + b t + c \in \SPQ[t]$ and
a monic real polynomial $M \in \R[t]$, both of degree two. Let $t_1$, $t_2 \in
\C$ be the two roots of $M = (t - t_1)(t - t_2)$. Denote by $R = P - M$ the
remainder polynomial of $P$ divided by $M$. Because of $M(t_1) = M(t_2) = 0$ we
have
\begin{equation*}
  P(t_1) = M(t_1) + R(t_1) = R(t_1)
  \quad\text{and}\quad
  P(t_2) = M(t_2) + R(t_2) = R(t_2)
\end{equation*}
and, provided that $t_1 \neq t_2$, the remainder $R$ is the unique interpolation
polynomial with respect to the interpolation data set $\{ (t_1, P(t_1)),$ $(t_2,
P(t_2)) \}$. Hence $R$ parameterizes the straight line $[P(t_1)] \vee [P(t_2)]$
or, if these two points coincide, the point $[P(t_1)] = [P(t_2)]$.

If $t_1 = t_2$ and thus $P(t_1) = P(t_2)$, the linear interpolation polynomial
is not well defined. Instead, the remainder polynomial $R$ describes the tangent
of the rational curve parameterized by $P$ at the point $[P(t_1)]$. In order to
see this, we compute
\begin{equation*}
  P(t_1) + P'(t_1) (t - t_1) = t_1^2 + b t_1 + c + (2 t_1 + b)(t - t_1) = (2 t_1 + b ) t + c - t_1^2.
\end{equation*}
It is equal to the remainder polynomial
\begin{align*}
  R = P - M = t^2 + b t + c - (t - t_1)(t - t_1) = (2 t_1 + b) t + c - t_1^2.
\end{align*}
Note that $P'(t_1)$ might vanish, that is, $b = -2 t_1$. In this case, the
parametric representation of the tangent as well as the remainder polynomial are
constant and equal to $c-t_1^2$.

In the context of the generic factorization algorithm the real polynomial $M$ is
one of the quadratic factors of the norm polynomial $P \Cj{P}$ and $t_1$, $t_2
\in \C$ are parameter values where the rational curve parameterized by $P$
intersects the null quadric $\N$. Hence, the remainder polynomial $R$
parameterizes the line $[P(t_1)] \vee [P(t_2)]$ spanned by these two intersection
points provided $[P(t_1)] \neq [P(t_2)]$. If these points are equal and $t_1
\neq t_2$, it is a linear parametrization of the single point $[P(t_1)] =
[P(t_2)]$. If, finally $t_1 = t_2$ (and hence also $P(t_1) = P(t_2)$), $R$
parameterizes the tangent of the rational curve $P$ in $[P(t_1)]$ (or again a
single point if $P'(t_1) = 0$).

\begin{defn}
  Consider a monic split quaternion polynomial $P = t^2 + b t + c \in \SPQ[t]$
  of degree two. Let $t_1$, $t_2$, $t_3$, $t_4 \in \C$ be the four roots of the
  norm polynomial $P \Cj{P} \in \R[t]$. We define the (at most six)
  \emph{remainder polynomials} of $P$ by $R_{ij} \coloneqq P - M_{ij}\in
  \SPQ[t]$ where $M_{ij} \coloneqq (t- t_i)(t-t_j) \in \R[t]$ for $i$, $j \in \{
  1,2,3,4 \}$ and $i < j$.
\end{defn}

Note that we only consider remainder polynomials that have \emph{real} split
quaternion coefficients, that is, we only use quadratic factors $M_{ij} \in
\R[t]$. The curve parameterized by $P$ intersects the null quadric $\N$ in four
points $[P(t_1)]$, $[P(t_2)]$, $[P(t_3)]$, $[P(t_4)] \in \P(\SPQ)$. Their
respective parameter values $t_1$, $t_2$, $t_3$, $t_4 \in \C$ are the four roots
of the norm polynomial $P \Cj{P}$. Hence, the polynomials $M_{ij} \in \R[t]$ are
the real quadratic factors of $P \Cj{P}$ and the remainder polynomials $R_{ij}
\in \SPQ[t]$ are the interpolation polynomials with respect to the interpolation
data sets $\{ (t_i, P(t_i)),$ $(t_j, P(t_j)) \}$. The interpolation polynomials
are defined in above's sense, i.e. they can be constant or, if $t_i = t_j$, may
parameterize the tangent of the curve at the point $[P(t_i)]$.

\begin{lem}
  \label{lem:remainder-polynomials}
  Let $P = t^2 + b t + c \in \SPQ[t]$ be a split quaternion polynomial and $R =
  r_1 t + r_0 \in \SPQ[t]$ be one of its remainder polynomials. If $R$ is of
  degree one, then $R$ has a either a unique root, $R$ parameterizes a null line
  or $r_1$ and $r_0$ are linearly dependent.
\end{lem}
\begin{proof}
  If $r_1 \Cj{r_1} \neq 0$, then $r_1^{-1} r_0 \in \SPQ$ is the unique root of
  $R$. Hence, we assume that $r_1 \Cj{r_1} = 0$. Moreover, we assume that $r_0$
  and $r_1$ are linearly independent, that is, $R$ parameterizes a straight line
  $\ell$ in $\P(\SPQ)$. In order to show that $\ell$ is a null line, we show
  that there are at least three intersection points between $\N$ and $\ell$. One
  of them is $[r_1] = [R(\infty)] \coloneqq
  [\lim_{t\to\infty}t^{-\deg{P}}P(t)]$.

  Let $M = (t-t_1)(t-t_2) \in \R[t]$ be the quadratic factor of $P \Cj{P}$ such
  that $P = M + R$. Further intersection points of $\N$ and $\ell$ are $[R(t_1)]
  = [P(t_1)]$ and $[R(t_2)] = [P(t_2)]$. Thus we have found three different
  intersection points unless $[R(t_1)] = [R(t_2)]$. Because $R$ is a linear
  polynomial with independent coefficients, equality of $[R(t_1)]$ and
  $[R(t_2)]$ implies $t_1 = t_2$ and $[P(t_1)] = [R(t_1)] = [R(t_2)] =
  [P(t_2)]$. As shown at the beginning of this subsection, $\deg R = 1$ and
  independence of $r_1$ and $r_0$ implies that this point is a regular point of
  the rational curve $P$ and $\ell$ is its tangent. We conclude that $\ell$ is
  also tangent to $\N$ in $[R(t_1)] = [R(t_2)]$. Since it also intersects $\N$
  in one further point $[r_1]$, it is a null line.
\end{proof}

\begin{rmk}
  If $P$ has a real linear factor $t - r \in \R[t]$, then $M = (t-r)^2$ is a
  quadratic factor of $P\Cj{P}$ and $t - r$ is a linear factor of the
  corresponding remainder polynomial $R$. In this case, the coefficients $r_1$
  and $r_0$ in Lemma~\ref{lem:remainder-polynomials} are linearly dependent.
  Conversely, linear dependency of $r_1$ and $r_0$ in
  Lemma~\ref{lem:remainder-polynomials} is equivalent to $R$ having a real root
  $r \in \R$. If $r$ is also a root of $M$, then $P$ has the real factor $t - r
  \in \R[t]$. In this case, factorizability is obvious whence we exclude it in
  the following.
\end{rmk}

\begin{thm}
  \label{th:factor-polynomial-monic-universal}
  The polynomial $P = t^2 + b t + c \in \SPQ[t]$ without a real factor admits a
  factorization if and only if there is a polynomial $R \in \SPQ[t]$ of degree
  one with linearly independent coefficients among the set of its remainder
  polynomials.
\end{thm}
\begin{proof}
  Without loss of generality we assume that the remainder polynomial $R_{12}$
  has degree one and its coefficients are linearly independent. If the leading
  coefficient of $R_{12}$ is invertible, $R_{12}$ has a unique root and then the
  generic factorization algorithm yields a factorization of $P$. If the leading
  coefficient is not invertible then $R_{12}$ parameterizes a null line by
  Lemma~\ref{lem:remainder-polynomials}. Provided that $R_{12}$ parameterizes a
  right ruling of $\N$, then we have seen in the proof of
  Theorem~\ref{th:factor-polynomial-monic-independent-coefficients} that $P$
  admits a factorization. If $R_{12}$ parameterizes a left ruling of $\N$, the
  complementary remainder polynomial $R_{34}$ parameterizes, again by the proof
  of Theorem~\ref{th:factor-polynomial-monic-independent-coefficients}, a right
  ruling and $P$ once more admits a factorization.

  Conversely assume that $P$ admits a factorization, that is $P = (t-h_1)
  (t-h_2)$ with $h_1$, $h_2 \in \SPQ$. Then
  \begin{equation*}
    \begin{aligned}
      P \Cj{P} = (t-h_1) (t-h_2) \Cj{((t-h_1) (t-h_2))} = (t-h_1) \Cj{(t-h_1)} (t-h_2) \Cj{(t-h_2)}
    \end{aligned}
  \end{equation*}
  and $M = (t-h_2) \Cj{(t-h_2)}$ is a quadratic factor of the norm polynomial $P
  \Cj{P}$. We compute the according remainder polynomial $R \in \SPQ[t]$ such
  that $P = M + R$. Since $t-h_2$ is a right factor of $M = \Cj{(t-h_2)}
  (t-h_2)$, it is also a right factor of $R = P - M$. Hence, there exists a
  split quaternion $r \in \SPQ$ such that $R = r (t - h_2)$. If $r = 0$, we have
  $P = M$ which contradicts the assumption that $P$ has no real factor. Thus,
  $R$ has degree one. In order to show independence of its coefficients we
  assume the opposite, i.e. there exists a real number $\alpha \in \R$ such that
  $0 = \alpha r - r h_2 = r (\alpha - h_2)$. Obviously $\alpha$ is a root of $R$
  and $[\alpha - h_2] \in \N$, that is $0 = (\alpha - h_2) \Cj{(\alpha - h_2)} =
  M(\alpha)$. Consequently, $P(\alpha) = M(\alpha) + R(\alpha) = 0$ and $P$ has
  the real factor $t-\alpha$ by Lemma~\ref{lem:zero-factor}.
\end{proof}

\begin{rmk}
  In our proof of Theorem~\ref{th:factor-polynomial-monic-universal} we appeal
  to the proof of
  Theorem~\ref{th:factor-polynomial-monic-independent-coefficients} whose
  assumptions are slightly different. This is admissible: The assumed
  independence of coefficients in
  Theorem~\ref{th:factor-polynomial-monic-independent-coefficients} implies
  independence of the coefficients of the remainder polynomial $R$. In
  Theorem~\ref{th:factor-polynomial-monic-universal}, this is not a conclusion
  but an assumption.
\end{rmk}

The considerations on remainder polynomials above allow to translate the
condition in Theorem~\ref{th:factor-polynomial-monic-universal}, that there be a
remainder polynomial of degree one, to the possibility to find an interpolation
data set $\{ (t_i, P(t_i)),$ $(t_j, P(t_j)) \}$ such that the according
interpolation polynomial parameterizes a real line. This is not possible
precisely if each interpolation polynomial parameterizes a point or a non-real
line and yields a profound geometrical interpretation of the equality and
inequality conditions in Theorem~\ref{th:factor-polynomial-b-real-c-nonreal} and
Theorem~\ref{th:factor-polynomial-b-nonreal-c-nonreal}, respectively. It also
clarifies the cause of non-factorizability in these theorems.

Although the norm polynomial $P \Cj{P}$ might have four distinct roots $t_1$,
$t_2$, $t_3$, $t_4 \in \C$, the line segment parameterized by $P$ in
Theorem~\ref{th:factor-polynomial-b-real-c-nonreal} or
Theorem~\ref{th:factor-polynomial-b-nonreal-c-nonreal} intersects $\N$ only at
two distinct (not necessarily real) points. Hence, two of the four points
represented by $P(t_1)$, $P(t_2)$, $P(t_3)$ and $P(t_4)$ coincide, respectively.
We set $P_\ell \coloneqq P(t_\ell)$ for $\ell \in \{1,2,3,4\}$ and, without loss
of generality, assume $[P_1] = [P_3]$ and $[P_2] = [P_4]$. The coefficients of
$P_1$ or $P_2$ might be non-real. If so, the points $[P_1]$ and $[P_2]$ are
complex conjugates, respectively, in the sense of
\cite[Section~5.1]{casas-alvero14}. At any rate, these two points are also the
intersection points of $\N$ and the line segment's support line, that is the
line $[1] \vee [c]$ in Theorem~\ref{th:factor-polynomial-b-real-c-nonreal} or
the line $[1] \vee [b]$ in
Theorem~\ref{th:factor-polynomial-b-nonreal-c-nonreal}.

\subsubsection*{Geometric interpretation of
  Theorem~\ref{th:factor-polynomial-b-real-c-nonreal}}
The curve parameterized by $P$ is a half-line with start-point $c$ and direction
$1$ in pseudo-Euclidean space $\SPQ$ and one of the two projective line segments
with endpoints $[1]$ and $[c]$ in $\P(\SPQ)$. All factorizability conditions of
Theorem~\ref{th:factor-polynomial-b-real-c-nonreal} pertain to~$c$ but their
geometric interpretation indirectly also depends on the projective point $[1]
\in \P(\SPQ)$ because of assumptions we made ``without loss of generality'' (in
particular monicity of~$P$).
\begin{itemize}
\item The sign of $c\Cj{c}$ distinguishes between points of the null quadric
  $\N$, its exterior, and its interior. The condition $c\Cj{c} \geq 0$ means,
  for example that the end point $[c]$ is on $\N$ or in the interior of $\N$.
\item The sign of $\Vector{c}\Cj{\Vector{c}}$ distinguishes between points of
  the ``asymptotic'' cone with vertex $[1]$ over the intersection of $\N$ with
  the ``ideal'' plane $c_0 = 0$, its exterior, and its interior. For example,
  points satisfying $c\Cj{c} \ge 0$ and $\Vector{c}\Cj{\Vector{c}} < 0$ lie
  inside the null quadric and outside its asymptotic cone.
\item If $\Vector{c}\Cj{\Vector{c}} > 0$, then the points $[P_1]$ and $[P_2]$
  are complex conjugates. The real line connecting them can be parameterized by
  the interpolation polynomial with respects to the data set $\{ (t_1,P_1),$
  $(t_2,P_2) \}$ since $t_1$ and $t_2$ can be chosen as complex conjugates from
  the set $\{ t_1,t_2,t_3,t_4 \}$.
\item The two endpoints separate the line $[1] \vee [c]$ into two line segments.
  The sign of $c_0$ determines which segment is parameterized by $P$, where
  $c_0 < 0$ denotes the one intersecting the ideal plane.
\item If $\Vector{c}\Cj{\Vector{c}} \leq 0$, then the points $[P_1]$ and $[P_2]$
  are real. If in addition $c \Cj{c} < 0$, then either $t_1$ and $t_2$ or $t_3$
  and $t_4$ are non-real. Hence, none of the interpolation polynomials
  parameterize a real line: Those involving the pairs $(t_1, t_3)$, $(t_1, t_4)$,
  $(t_2, t_3)$ or $(t_2, t_4)$ parameterize non-real lines and those involving
  the pairs $(t_1,t_2)$ or $(t_3,t_4)$ are constant.
\item If $\Vector{c}\Cj{\Vector{c}} \leq 0$ and $c \Cj{c} \geq 0$ the sign of
  $c_0$ is crucial. Similar to the item above, $c_0 \geq 0$ yields only
  interpolation polynomials parameterizing a non-real line or a point.
  Conversely, $c_0 < 0$ implies that $t_1$, $t_2$, $t_3$ and $t_4$ are real and
  we can find interpolation polynomials parameterizing a real line, e.g. the one
  according to the data set $\{ (t_1,P_1),$ $(t_2,P_2) \}$.
 \end{itemize}

 With exception of the sign of $c_0$, all inequality conditions are of
 projective nature. In the affine space $\SPQ$, the sign of $c_0$ distinguishes
 between half-spaces. A natural framework for a unified geometric interpretation
 of all inequalities is \emph{oriented projective geometry}
 \cite{stolfi87,kirby02}.

\subsubsection*{Geometric interpretation of Theorem~\ref{th:factor-polynomial-b-nonreal-c-nonreal}}
Again, the polynomial $P$ parameterizes a projective line segment or possibly a
projective line in $\P(\SPQ)$. The endpoints of the line segment are $[e_{1,2}] =
[\pm 2 \mu \sqrt{\lambda + \mu^2} + 2 (\lambda + \mu^2) \mp b \sqrt{\lambda +
  \mu^2}]$. At any rate, the segment contains the point~$[1]$.
\begin{itemize}
\item Similar as in Theorem~\ref{th:factor-polynomial-b-real-c-nonreal}, the
  sign of $b\Cj{b} = \Vector{b} \Cj{\Vector{b}}$ distinguishes between lines or
  line segments with supporting line that lies on the null quadric's asymptotic
  cone or in the cone's interior/exterior.
\item The case $b \Cj{b} > 0$ is identical to the case $\Vector{c} \
  \Cj{\Vector{c}} > 0$ in Theorem~\ref{th:factor-polynomial-b-real-c-nonreal}.
\item If $b \Cj{b} = 0$ or $b \Cj{b} < 0$, then $[P_1]$ and $[P_2]$ are real.
  The respective conditions $\lambda < 0$ or $b \Cj{b} + 4 \lambda \le 0$ and $b
  \Cj{b} + 4 \lambda \le 4 \mu \sqrt{- b \Cj{b}} \le -(b \Cj{b} + 4 \lambda)$
  ensures that $t_1$, $t_2$, $t_3$ and $t_4$ are real as well whence there
  exists an interpolation polynomial parameterizing a real line. Otherwise, the
  parameter values $t_1$, $t_2$, $t_3$, $t_4$ are non-real and all interpolation
  polynomials parameterize non-real lines or points.
\item Because $1$ and $b$ are linearly independent, the condition $\lambda +
  \mu^2 = 0$ is necessary and sufficient for the existence of a real zero of
  $P$. This is equivalent to $P$ being a linear parametrization of the line $[1]
  \vee [b]$ multiplied with a linear real polynomial. This is a trivial case
  which we have excluded.
\end{itemize}

Based on these considerations we can state a simple geometric criterion for the
existence of factorizations in case of monic polynomials with dependent
coefficients. This covers all polynomials $P = at^2 + bt + c$ in
\eqref{eq:polynomial} where $a$, $b$ and $c$ are linearly dependent and the
leading coefficient $a$ is invertible since multiplication with $a^{-1}$ yields
a monic polynomial. Moreover, if $a$, $b$ and $c$ are linearly dependent so are
$1$, $a^{-1}b$ and $a^{-1}c$ and vice versa. In fact we even covered those cases
where the leading coefficient $a$ of $P$ is not invertible but the curve
parameterized by $P$ is not contained in the null quadric $\N$, i.e. the norm
polynomial does not vanish. As long as there is a point on the curve which is
not contained in $\N$, one can apply a proper parameter transformation to $P$
such that the leading coefficient becomes invertible. Factorizability of the
thus obtained polynomial guarantees factorizability of the initial one.

\begin{thm}
  \label{th:geometric-characterization}
  Assume that the polynomial $P \in \SPQ[t]$ is of degree two, has linearly
  dependent coefficients, no real factor of positive degree, and a non-vanishing
  norm polynomial. Denote by $L$ the vector sub-space (of dimension two) spanned
  by the coefficients of $P$. There exists a factorization of $P$ if and only if
  the point sets $\{[P(t)] \mid t \in \R \cup \{\infty\}\}$ (line segment
  parameterized by $P$) and $[L]$ intersect $\N$ in the same number of points.
\end{thm}

The content of Theorem~\ref{th:geometric-characterization} is visualized in
Figure~\ref{fig:gfigure}. Images in the first and second row refer to the
geometric interpretation of Theorem~\ref{th:factor-polynomial-b-real-c-nonreal},
the last row refers to Theorem~\ref{th:factor-polynomial-b-nonreal-c-nonreal}.
Images in the first row and the first and second image in the last row
correspond to cases that admit factorizations. All other images correspond to
cases that don't.

\begin{figure}
  \centering
    \includegraphics{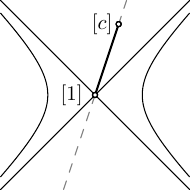}\qquad
    \includegraphics{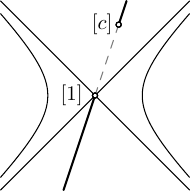}\qquad
    \includegraphics{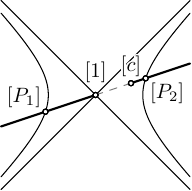}\\[3ex]
    \includegraphics{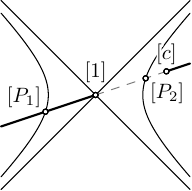}\qquad
    \includegraphics{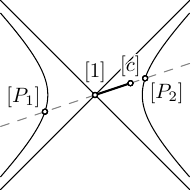}\qquad
    \includegraphics{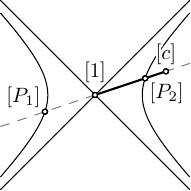}\\[3ex]
    \includegraphics{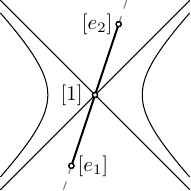}\qquad
    \includegraphics{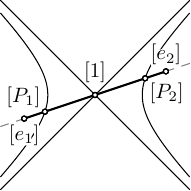}\qquad
    \includegraphics{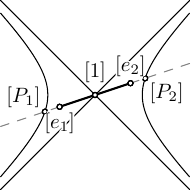}
\caption{Geometric interpretation of factorizability in case of dependent
      coefficients.}
    \label{fig:gfigure}
\end{figure}

\subsubsection*{Geometric Interpretation of Theorem~\ref{th:factor-polynomial-monic-independent-coefficients}}

In this, the coefficients of $P$ are independent whence it parameterizes a
(regular) conic section $\mathcal{C}$ in $\P(\SPQ)$. It intersects the null
quadric $\N$ in four points, not necessarily real or distinct. Nonetheless, a
suitable choice of a remainder polynomial (which corresponds to a suitable
choice of a line) is always possible: We may connect a generic pair of distinct
real intersection points, a pair of conjugate complex intersection points or
pick the tangent in a real intersection point of multiplicity at least two. The
most interesting case is that of $\mathcal{C}$ lying in a tangent plane of $\N$.
In this case, the intersection of $\mathcal{C}$ and $\N$ will always contain a
left and a right ruling. In the proof of
Theorem~\ref{th:factor-polynomial-monic-independent-coefficients} we have shown
that right ruling is always a suitable choice.

\begin{figure}
  \centering
  \includegraphics{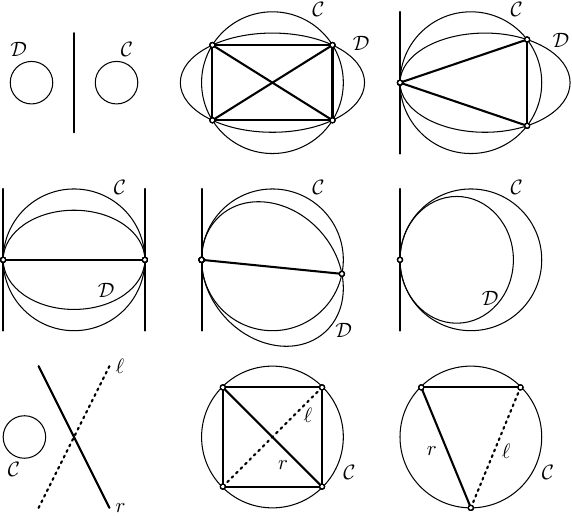}
  \caption{Geometric interpretation of
    Theorem~\ref{th:factor-polynomial-monic-independent-coefficients}.}
  \label{fig:conic-pencils}
\end{figure}

This is illustrated in Figure~\ref{fig:conic-pencils}. In the first and second
row, we assume that $\mathcal{C}$ is not in a tangent plane of $\N$ so that we
actually look at the intersection of two regular conics, $\mathcal{C}$ and the
intersection conic $\mathcal{D}$ of $\N$ with the plane of $\mathcal{C}$. For
diverse relative positions of $\mathcal{C}$ and $\mathcal{D}$, suitable
connecting lines are drawn in bold line style. In the top-left image, all
intersection points are non-real but a real connecting line does exist.

The bottom row illustrates cases where $\mathcal{C}$ is in a tangent plane of
$\N$. Thus, the plane of $\mathcal{C}$ intersects $\N$ in a left ruling $\ell$
and a right ruling $r$. Once again, suitable choices of lines are drawn in bold
and potentially invalid lines in dotted line style.

\subsection{Factorization of Polynomials with Non-Invertible Leading
  Coefficient}
\label{sec:factor-non-invertible}

So far, we considered polynomials $P = at^2 + bt + c$ in \eqref{eq:polynomial}
with invertible leading coefficient and non-vanishing norm polynomial. Taking
into account the already explained possibility of re-parameterization just
before Theorem~\ref{th:geometric-characterization}, the only missing case in our
discussion so far is that the curve parameterized by $P$ is contained in $\N$.
This is the case if and only if the norm polynomial vanishes:
\begin{equation}
  \label{eq:normpolynomial}
  P \Cj{P} = a \Cj{a} t^4 + (a \Cj{b} + b \Cj{a}) t^3 + (a \Cj{c} + c \Cj{a} + b \Cj{b}) t^2 + (b \Cj{c} + c \Cj{b}) t + c \Cj{c} = 0.
\end{equation}
It will turn out that factorizations always exist. In our investigation, we
distinguish two cases:
\begin{itemize}
\item $a$, $b$, $c \in \SPQ$ linearly dependent $\leadsto$
  Theorem~\ref{th:factor-polynomial-nonmonic-dependent-coefficients}
\item $a$, $b$, $c \in \SPQ$ linearly independent $\leadsto$
  Theorem~\ref{th:factor-polynomial-nonmonic-independent-coefficients}
\end{itemize}

We start the discussion by assuming that the coefficients $a$, $b$ and $c$ are
linearly dependent. Hence, the corresponding points $[a]$, $[b]$ and $[c]$ lie
on a line. Since the norm polynomial vanishes this line is a ruling of the null
quadric $\N$.

\begin{thm}
  \label{th:factor-polynomial-nonmonic-dependent-coefficients}
  The polynomial $P = a t^2 + b t + c \in \SPQ[t]$ with $P \Cj{P} = 0$ and
  linearly dependent coefficients admits a factorization.
\end{thm}

In the ensuing proof of
Theorem~\ref{th:factor-polynomial-nonmonic-dependent-coefficients} it is
possible that the coefficients $b$ or $c \in \SPQ$ vanish whence the points
$[b]$ and $[c] \in \P(\SPQ)$ become undefined. For the sake of readability, we
do not always take into account this possibility in our proof which,
nonetheless, is also valid for these special cases.

\begin{proof}[Proof of Theorem~\ref{th:factor-polynomial-nonmonic-dependent-coefficients}]
  Without loss of generality, we may assume that $[a]$, $[b]$ and $[c]$ lie on a
  right ruling since $P$ is factorizable if and only if its conjugate $\Cj{P} =
  \Cj{a} t^2 + \Cj{b} t + \Cj{c}$ is factorizable and conjugation exchanges
  right rulings and left rulings. The proof is subdivided into four different
  cases.

  For the first two cases, we assume that $b = 0$. If $c$ is a real multiple of
  $a$, which is equivalent to $[a] = [c]$, then $P$ can be written as $P = a
  \hat{P}$ with $\hat{P} \in \R[t]$. Since $\hat{P}$ admits a factorization in
  terms of Lemma~\ref{lem:factor-polynomial-b-real-c-real}, so does $P$.

  If $[a] \neq [c]$ then $[a] \vee [c]$ is a right ruling of $\N$. Due to
  Theorem~\ref{th:affine-two-plane-of-zeros} there exists a split quaternion $h$
  such that $c = a h$. We write $P = a (t^2 + h)$ and show that $t^2 + h$ admits
  a factorization. There are two degrees of freedom in the choice of $h$, namely
  the two real parameters $\lambda$ and $\mu$ of
  Theorem~\ref{th:affine-two-plane-of-zeros}. We set them both equal to
  zero. Then $h = h_0 + h_1 \qi$ is an element of the sub-ring $\langle 1, \qi
  \rangle_\R \subset \SPQ$. The assumption $[a] \neq [c]$ yields $h_1 \neq 0$.
  We have $h \Cj{h} = h_0^2 + h_1^2 > 0$ and $\Vector{h} \Cj{\Vector{h}} = h_1^2
  > 0$. Hence, $t^2 + h$ is factorizable by means of
  Theorem~\ref{th:factor-polynomial-b-real-c-nonreal}.

  We are left with the case that $b \neq 0$ and $[a]$, $[b]$, $[c]$ lie on a
  right ruling of $\N$.

  If $[a] = [b]$ we may write $b = \alpha a$ with $\alpha \in \R \setminus
  \{0\}$. Moreover, by Theorem~\ref{th:affine-two-plane-of-zeros}, there exist
  $h = h_0 + h_1\qi$ such that $c = ah$ whence $P = a(t^2 + \alpha t + h)$. By a
  suitable parameter transformation, we can eliminate the coefficient of $t$ in
  $t^2 + \alpha t + h$ while preserving the vector part $\Vector{h}$ of the
  constant coefficient. Existence of a factorization is once more guaranteed by
  Theorem~\ref{th:factor-polynomial-b-real-c-nonreal}.

  If $[a ] \neq [b]$, we chose, again according to
  Theorem~\ref{th:affine-two-plane-of-zeros}, $h = h_0 + h_1 \qi$ such that $b =
  a h$. Moreover, there exist $\alpha, \beta \in \R$ such that $c = \alpha a +
  \beta b$ due to linear dependency of $a$, $b$ and $c$. We can write $P = a
  (t^2 + h t + \alpha + \beta h)$ and since $h \Cj{h} > 0$, the polynomial $t^2
  + h t + \alpha + \beta h$ fulfills the condition of
  Theorem~\ref{th:factor-polynomial-b-nonreal-c-nonreal} to be factorizable.
\end{proof}

\begin{rmk}
  In the proof of Theorem
  \ref{th:factor-polynomial-nonmonic-dependent-coefficients}, we set the two
  real parameters $\lambda$ and $\mu$ equal to zero. But there are infinitely
  many possible choices for these parameters according to
  Theorem~\ref{th:affine-two-plane-of-zeros}. The crucial ingredient in the
  proof is that the norm of $h$ or $\IM h$ is strictly positive. Since these
  norms depend continuously on $\lambda$ and $\mu$, strict positivity is
  preserved for infinitely many choices of $\lambda$ and $\mu$. Hence, there are
  infinitely many factorizations for the second, third and fourth case. In the
  first case ($b = 0$, $[a] = [c]$) existence of infinitely many factorizations
  is guaranteed by Corollary~\ref{cor:factor-polynomial-b-real-c-real}.
\end{rmk}

Finally, we present the missing factorization result for split quaternion
polynomials with vanishing norm polynomial and independent coefficients.

\begin{thm}
  \label{th:factor-polynomial-nonmonic-independent-coefficients}
  The polynomial $P = a t^2 + b t + c \in \SPQ[t]$ with $P \Cj{P} = 0$, $a \neq
  0$ and linearly independent coefficients admits a factorization.
\end{thm}
\begin{proof}
  The condition $P \Cj{P} = 0$ implies that each coefficient in Equation
  \eqref{eq:normpolynomial} vanishes. In particular, we have
  \begin{equation}
    \label{eq:normpolynomial-quadratic-coefficient}
    \begin{gathered}
    0 = b \Cj{a} + a \Cj{b} = 2 \, q(b,a) = 2 \, q(\Cj{b},\Cj{a}) = \Cj{b} a + \Cj{a} b,\\
    0 = b \Cj{c} + c \Cj{b} = 2 \, q(b,c) = 2 \, q(\Cj{b},\Cj{c}) = \Cj{b} c + \Cj{c} b.
    \end{gathered}
  \end{equation}
  If $b$ is not invertible, i.e. $b \Cj{b} = 0$,
  Equation~\eqref{eq:normpolynomial-quadratic-coefficient} implies that not only
  the points $[a]$, $[b]$, $[c] \in \P(\SPQ)$ are contained in $\N$, but also the
  lines $[a] \vee [b]$, $[b] \vee [c]$. This is not possible because $P$
  parameterizes a non-singular planar section of~$\N$.

  Hence $b$ is invertible and the split quaternion $h \coloneqq -b^{-1} c$ is a
  zero of $P$:
  \begin{align*}
    P(h) &= a h^2 + b h + c = a (- b^{-1} c)^2 + b (- b^{-1} c) + c = \\
         &= \frac{1}{(b \Cj{b})^2} a \Cj{b} c \Cj{b} c \overset{\eqref{eq:normpolynomial-quadratic-coefficient}}{=} - \frac{1}{(b \Cj{b})^2} a \Cj{b} \underbrace{c \Cj{c}}_{=0} b = 0.
  \end{align*}
  By Lemma~\ref{lem:zero-factor}, $t - h$ is a right factor of $P$ and a
  factorization exists.
\end{proof}

Theorem~\ref{th:factor-polynomial-nonmonic-independent-coefficients} in
combination with
Theorem~\ref{th:factor-polynomial-monic-independent-coefficients} or
Theorem~\ref{th:factor-polynomial-nonmonic-dependent-coefficients} implies a
corollary each.
\begin{cor}
  \label{cor:independent}
  Any quadratic split quaternion polynomial with linearly independent
  coefficients admits a factorization.
\end{cor}

\begin{cor}
  \label{cor:vanishing}
  Any quadratic split quaternion polynomial with vanishing norm admits a
  factorization.
\end{cor}

\section{Future Research}

We have presented a complete discussion of factorizability of quadratic
polynomials over the split quaternions and provided a geometric interpretation
in the (oriented) projective space over the split quaternions. A natural next
step is, of course, factorizability questions for higher degree polynomials. 
We expect to be able to re-use ideas and techniques of this paper. One thing
that is already clear is existence of non-factorizable polynomials of arbitrary
degree.

Other questions of interest include factorization results for different
algebras. One obstacle to generalizations is the lack of a suitable substitute
of quaternion conjugation, that is, a linear map that gives inverse elements up
to scalar multiples. Existence of such a map and its exploitation for
factorization on suitable and interesting sub-algebras are on our research
agenda as well. Preliminary results in Conformal Geometric Algebra already
exist.

\section*{Acknowledgment}

Johannes Siegele was supported by the Austrian Science Fund (FWF): P~30673 (Extended Kinematic Mappings and Application to Motion Design).
Daniel F. Scharler was supported by the Austrian Science Fund (FWF): P~31061 (The Algebra of Motions in 3-Space).
 
\bibliographystyle{elsarticle-harv}

\end{document}